\documentclass[11pt]{article}

\usepackage{notoccite}

\usepackage{epstopdf}% To incorporate .eps illustrations using PDFLaTeX, etc.
\usepackage[caption=false]{subfig}% Support for small, `sub' figures and tables

\usepackage[longnamesfirst,sort]{natbib}% Citation support using natbib.sty
\bibpunct[, ]{(}{)}{;}{a}{,}{,}% Citation support using natbib.sty
\usepackage{tikz}
\usetikzlibrary{matrix}
\usepackage{calligra} 

\usepackage{multirow}
\usepackage{caption}
\usepackage{amsmath}
\usepackage{amsfonts}
\usepackage{amssymb}
\usepackage{graphicx}
\usepackage{graphics}
\usepackage{amscd,amsmath,amstext,amsfonts,amsbsy,amssymb,amsthm}
\usepackage[unicode,colorlinks,linkcolor= blue]{hyperref}
\usepackage{nccmath}
\usepackage{framed}
\usepackage{authblk}
\usepackage{bbm}
\usepackage{listings}
\usepackage{pdfpages}
\graphicspath{ {images/} }
\usepackage{tikz}
\usetikzlibrary{arrows,positioning,shapes.geometric}

\newtheorem{thm}{Theorem}[section]

\newtheorem{proposition}[thm]{Proposition}

\newtheorem{defi}[thm]{Definition}

\newcommand{\bequ}{\begin{equation}}
\newcommand{\eequ}{\end{equation}}

\newcommand{\norm}[1]{\left\lVert#1\right\rVert}
\begin{document}

	%--------------------------------------------------------------------------
	
	\begin{center}
		{\Large
			{\sc  \bf{Distribution regression model with a Reproducing Kernel Hilbert Space approach}}
		}
		\bigskip
		
		Bui Thi Thien Trang $^{1}$ \& Jean-Michel Loubes $^{1}$ \& Laurent Risser $^{1}$ \& Patricia Balaresque $^{2}$ 
		\bigskip
		
		{\it
			$^{1}$ Institut de Math\'ematiques de Toulouse \\
			Universit\'e Paul Sabatier 118, route de Narbonne F-31062 Toulouse Cedex 9\\
			(tbui, jean-michel.loubes, laurent.risser)@math.univ-toulouse.fr
			
			$^{2}$ Laboratoire d'Anthropologie Mol\'eculaire et Imagerie de Synth\`ese (AMIS)\\
			Facult\'e de M\'edecine Purpan, 37 all\'ees Jules Guesde, Toulouse \\
			patricia.balaresque@univ-tlse3.fr
		}
	\end{center}
	\bigskip
	
%--------------------------------------------------------------------------------------------------------

	{\bf Abstract.} In this paper, we introduce a new distribution regression model for probability distributions. 
	This model is based on a Reproducing Kernel Hilbert Space (RKHS) regression framework, where universal kernels are built using Wasserstein distances for distributions belonging to $\mathcal{W}_2(\Omega)$ and  $\Omega$ is a compact subspace of $\mathbb{R}$. 
	We prove the universal kernel property of such kernels and use this setting to perform regressions on functions.  
	Different regression models are first compared with the proposed one on simulated functional data for both one-dimensional and two-dimensional distributions. We then apply our regression model to transient evoked otoascoutic emission (TEOAE) distribution responses and real predictors of the age.  
	\smallskip
	
	{\bf Keywords.} Regression, Reproducing kernel Hilbert space,  Wasserstein distance, Transient evoked otoscoutic emission.
	
%--------------------------------------------------------------------------------------------------------

\section{Introduction}\label{sec1}
	Regression analysis is a predictive modeling technique that has been widely studied over the last decades with the goal to investigate relationships between predictors and responses (inputs and outputs) in regression models, see for instance  \cite{kutner2004applied}, \cite{azais2006modele} and references therein. When the inputs belong to functional spaces, different strategies have been investigated and used in several application domains about functional data analysis \cite{neter1996applied}, \citet{ramsay2007applied}. The Reproducing Kernel Hilbert Space (RKHS) framework became recently popular to extend the results of the statistical learning theory in the context of regression of functional data as well as to develop estimation procedures of functional valued functions $f$ as in~\cite{preda2007regression, kadri2010nonlinear} for instance. This framework is particularly important in the field of statistical learning theory because of the so-called \emph{Representer theorem}, which states that every function can be written as a linear combination of the kernel function evaluated at training points \cite{berlinet2011reproducing}, providing a setting to estimate functions over a wide range of functional spaces.
	
	In our framework, we aim to solve the regression problem with inputs belonging to probability distribution spaces whose outputs are real values. Specially, we consider the model
	\begin{equation}\label{equa1}
	y_i = f(\mu_i)+\sigma \epsilon_i,
	\end{equation}
	 where $\{\mu_i\}_{i=1}^n$ are probability distributions on $\mathbb{R}$, $\{y_i\}_{i=1}^n$ are real numbers and the $\epsilon_i$ represent an independent and identically distributed Gaussian noise and $\sigma$ is the level of noise. As in classical regression models, this setting estimates an unknown function $f$ from the observations $\{\left(\mu_i, y_i\right) \}_{i=1}^n$.
	
	Distribution regression problem has become a major concern in the recent years. A common approach  is to  look at embedding  into a Hilbert space  using kernels as for instance in \cite{smola2007hilbert} with the mean embedding or see in~\cite{muandet2012learning} for a review. The studies for regression with distribution inputs and real-valued responses are one of the most popular research topics in this field, for instance \cite{poczos2013distribution} and \cite{oliva2014fast}. For more information about distribution regression theoretical properties, we refer to \cite{szabo2016learning} and \cite{muandet2017kernel}.
	
	In what follows, we will consider kernels built using the Wasserstein distance. (see in~\cite{villani2008optimal}). Some kernels with this metric have been developed in  \citet{KolouriZouRohde,peyre2016gromov}. In studies of \cite{KolouriZouRohde}, authors built a sliced Wasserstein kernel based on Wasserstein distance, which was proved be a positve definite kernel for absolutely continuous distributions. More generally, for any distributions, authors in \citet*{bachoc2017gaussian} built a family of positive definite kernel based also on Wasserstein distance for distributions on the real line. The properties of characteristic kernels were explored further in \citet{sriperumbudur2008injective, sriperumbudur2010hilbert, sriperumbudur2011universality}. Besides that, a closely related to these kernel for achieving the richness of RKHS was previously studied through the so-called universal kernels in \cite{steinwart2001influence}. Universality property for kernels is stronger than the characteristic property. Actually, using results in~\cite{sriperumbudur2010relation} all universal kernels are characteristic. Within the framework of \cite{christmann2010universal}, the authors prove universality of Gaussian-type RBF-kernel. Futher, combining the positive definite kernel built by \citet*{bachoc2017gaussian} with the Gaussian-type RBF-kernel in \cite{christmann2010universal}, we consider a universal kernel based on the Wasserstein distance between two probability distributions on a compact space.

	 The paper falls into the following parts. In Section~\ref{sec2}, we first recall important concepts about kernels on Wasserstein spaces $\mathcal{W}_2(\mathbb{R})$. Section~\ref{sec3} deals with the proposed setting of distribution regression models.  We then assess the numerical performance of this method in Section~\ref{sec4}. The tests are first performed on simulated generated data in one and two dimensions to compare our model with state-of-the-art ones. Then, we study the relationship between the age and hearing sensitivity by using TEOAEs recording that are acquired by stimulating with a very short but strong broadband stimulus. These recordings are then the ear responses by emitting a sound track on a given frequency. More precisely, we predict the age of the subjects, on which they were acquired using the proposed distribution regression model, from TEOAE data. Discussions are finally drawn Section~\ref{sec5}.
	
%--------------------------------------------------------------------------------------------------------

\section{Kernel on Wasserstein space $\mathcal{W}_2(\mathbb{R})$}\label{sec2}
\subsection{The Wasserstein space on $\mathbb{R}$} \label{2.1}
	Let us consider the set $\mathcal{W}_2(\mathbb{R})$ of probability measures on $\mathbb{R}$ with a finite moment of order two. For two $\mu, \nu$ probability distributions in $\mathcal{W}_2(\mathbb{R})$, we denote $\Pi(\mu,\nu)$ the set of all probability measures $\pi$ over the product set $\mathbb{R}\times \mathbb{R}$ with first (resp. second) marginal $\mu$ (resp. $\nu$).\\
	The transportation cost with quadratic cost function, which we denote quadratic transportation cost, between two measures $\mu$ and $\nu$ is defined as:
	
	\begin{equation}\label{cost}
	\mathcal{T}_2(\mu,\nu)=\underset{\pi \in \Pi(\mu,\nu)}{\text{inf}} \int |x-y|^2d\pi (x,y).
	\end{equation}
	This transportation cost allows to endow the set $\mathcal{W}_2(\mathbb{R})$ with a metric by defining the quadratic Monge-Kantorovich (or quadratic Wasserstein)   distance between $\mu$ and $\nu$ as
	
	\begin{equation*}
	W_2(\mu,\nu) = \mathcal{T}_2(\mu,\nu)^{1/2}.
	\end{equation*}
	A probability measure $\pi$ in $\Pi(\mu,\nu)$ performing the infimum in (\ref{cost}) is called an optimal coupling. This vocabulary transfers to a random vector $(X_1,X_2)$ with distribution $\pi$. We will call $\mathcal{W}_2(\mathbb{R})$ endowed with the distance $W_2$ the Wasserstein space. More details on Wasserstein spaces and their links with optimal transport problems can be found in \cite{villani2008optimal}.\\
  
	For distributions in $\mathbb{R}$, the Wasserstein distance can be written in a simpler way as follows: For any $\mu \in \mathcal{W}_2(\mathbb{R})$, we denote by $F^{-1}_\mu$ the quantile function associated to $\mu$. Given a uniform random variable $U$ on $[0,1]$, $F^{-1}_\mu (U)$ is the random variable with law $\mu$. Then,  for every $\mu$ and $\nu$  the random vector $(F^{-1}_\mu (U), F^{-1}_\nu (U))$ is the optimal coupling (see \cite{villani2008optimal}), where $F^{-1}$ is defined as
	\begin{equation}\label{quant}
	F^{-1}_\mu (t) = \text{inf}\{u,F_\mu(u) \geq t \}.
	\end{equation}
	In this case, the simplest expression for the Wasserstein distance is  given in \cite{whitt1976bivariate}:
	\begin{equation}
	W_2(\mu,\nu) = \mathbb{E}(F^{-1}_\mu(U) - F^{-1}_\nu(U))^2.
	\end{equation}

	Topological properties of Wasserstein spaces are reviewed  in \cite{villani2008optimal}. Hereafter, compacity will be required and will be obtained as follows: let    $\Omega \subset \mathbb{R}$ be a compact subset, then the Wasserstein space $\mathcal{W}_2(\Omega)$ is also compact. In this paper, we consider Wasserstein spaces $\mathcal{W}_2(\Omega)$, where $\Omega$  is a compact subset on $\mathbb{R}$ endowed with the Wasserstein distance $W_2$. Hence for any $\mu\in \mathcal{W}_2(\Omega)$, we denote $F_\mu|_\Omega : \Omega \rightarrow [a,b]$ with $[a,b]\subset [0,1]$ the distribution function restricted on a compact subset $\Omega$. We also define $F^{-1}_\mu|_\Omega$ as:
	\begin{equation}\label{quant2}
	F_\mu^{-1}|_\Omega(t)= inf\{u\in \Omega,F_\mu|_\Omega(u) \geq t \}, \, \forall t\in [a,b].
	\end{equation}
	Given a uniform random variable $V$ on $[a,b]$, $F_\mu^{-1}|_\Omega$ is a random variable with law $\mu$. By inheriting properties from $\mathcal{W}_2(\mathbb{R})$ for every $\mu$ and $\nu$ in $\mathcal{W}_2(\Omega)$, the random vector $\left(F_\mu^{-1}|_\Omega(V), F_\nu^{-1}|_\Omega(V) \right)$ is an optimal coupling. In this case, we consider in this paper the simplest expression for the Wasserstein distance between $\mu$ and $\nu$ in $\mathcal{W}_2(\Omega)$:
	\begin{equation}\label{wassrdist}
	W_2(\mu,\nu) = \mathbb{E}(F_\mu^{-1}|_\Omega(V) - F_\nu^{-1}|_\Omega(V))^2.
	\end{equation}

	\subsection{Kernel}
	Constructing a positive definite kernel defined on the Wasserstein space is not obvious and  was recently done in \cite{bachoc2017gaussian}. For sake of completeness, we recall here briefly this construction. 
	\begin{thm}\label{exp.mono}
		Let $k_\Theta: \mathcal{W}_2(\Omega) \times \mathcal{W}_2(\Omega) \rightarrow \mathbb{R}$ with the parameter $\Theta:= (\gamma, H,l)$ such that $\gamma \neq 0$ and $l>0$ defined as
		\begin{equation}\label{kernel}
		k_\Theta(\mu,\nu) := \gamma^2 \exp\left(-\frac{W_2^{2H}(\mu,\nu)}{l} \right).
		\end{equation}
	Then for 	$0 < H\leq 1$, $k_\Theta$	is a positive definite kernel. 
	\end{thm}
	The proof of this Theorem directly follows Theorem~\ref{scho} and Propositions~\ref{theoW}. In this paper we use Theorem~\ref{exp.mono} to study the properties of such kernel in the RKHS regression framework. 
  
	\begin{thm} \label{scho}
		(\textbf{Schoenberg}) Let $F: \mathbb{R}^{+} \rightarrow \mathbb{R}^{+}$ be a completely monotone function, and $K$ a negative definite kernel. Then $(x,y) \mapsto F(K(x,y))$ is a positive definite kernel.
	\end{thm}
\begin{proof}
	We refer to \cite{cowling1983harmonic} for more details on completely monotone functions and a proof of Theorem \ref{scho}. We use this result to give a proof of the positive definite kernel which is presented in Theorem \ref{exp.mono}.
\end{proof}
	
  The following proposition which can be found a general version for any distribution in $\mathcal{W}_2(\mathbb{R})$ in \cite{bachoc2017gaussian}, here we present again a result for our case in $\mathcal{W}_2(\Omega)$. Finally we give conditions on the exponent $H$ to achieve the negative definite kernel using  exponents of the Wasserstein distance.
	\begin{proposition}\label{theoW}
		The function $W^{2H}_2$ is a negative definite kernel if and only if  $0 < H\leq 1$.
	\end{proposition}
	
	\begin{proof}
		One can find in \cite{KolouriZouRohde} a version of this result for absolutely continuous distributions in $\mathcal{W}_2(\Omega)$. The proof of this proposition is proved for any distribution in $\mathcal{W}_2(\Omega)$.\\
		For any $\mu \in \mathcal{W}_2(\Omega)$, we denote by $F_\mu^{-1}|_\Omega $ the function defined by (\ref{quant2}).\\
		From statements in Section \ref{2.1}, we well known that given a uniform random variable $V$ on $\left[a,b\right]$, $F_\mu^{-1}|_\Omega(V)$ is a random variable with a law $\mu$, and furthermore for every $\mu$ and $\nu$ in $\mathcal{W}_2(\Omega)$,
		\begin{equation*}
		W_2(\mu,\nu)=\mathbb{E}(F_\mu^{-1}|_\Omega(V) - F_\nu^{-1}|_\Omega(V))^2,
		\end{equation*}
		that is to say the coupling of $\mu$ and $\nu$ given by the random vector $\left(F_\mu^{-1}|_\Omega(V), F_\nu^{-1}|_\Omega(V) \right)$ is optimal.\\
		Now let us consider every $\mu_1,\cdots,\mu_n \in \mathcal{W}_2(\Omega)$ and $c_1,\cdots,c_n\in \mathbb{R}$ such that $\sum_{i=1}^{n}c_i = 0$. We have
		\begin{align*}
		\sum_{i=1}^{n}c_ic_j W_2^2(\mu_i,\mu_j) &= \sum_{i,j=1}^{n}c_ic_j \mathbb{E}(F_{\mu_i}^{-1}|_\Omega(V) - F_{\mu_j}^{-1}|_\Omega(V))^2 \\
		{} &= \sum_{i,j=1}^{n}c_ic_j \mathbb{E} \left(F_{\mu_i}^{-1}|_\Omega(V) \right)^2 + \sum_{i,j=1}^{n}c_ic_j \mathbb{E} \left(F_{\mu_j}^{-1}|_\Omega(V) \right)^2 \\
		{}&- 2\sum_{i,j=1}^{n}c_ic_j \mathbb{E}\left( F_{\mu_i}^{-1}|_\Omega(V)F_{\mu_j}^{-1}|_\Omega(V)\right).
		\end{align*}
		Using $\sum_{i=1}^{n}c_i = 0$ the first two sums vanish and we obtain
		\begin{align*}
		\sum_{i=1}^{n}c_ic_j W_2^2(\mu_i,\mu_j) &= - 2\sum_{i,j=1}^{n}c_ic_j \mathbb{E}\left( F_{\mu_i}^{-1}|_\Omega(V)F_{\mu_j}^{-1}|_\Omega(V)\right)\\
		{} &= -2\mathbb{E}\left(\sum_{i=1}^{n}c_iF_{\mu_i}^{-1}|_\Omega(V)  \right)^2 \leq 0.
		\end{align*}
		Thus $W^{2H}_2$ is a negative definite kernel for $0<H\leq 1$.
	\end{proof}
	Now, we give the proof for Theorem \ref{exp.mono} by using results from both of Schoenberg theorem and proposition \ref{theoW}.
	\begin{proof}
		 The proof of Theorem \ref{exp.mono} follows immediately below from Theorem \ref{scho} and Proposition \ref{theoW}.  Applying Proposition \ref{theoW}, we deduce that $W_2^{2H}(\mu,\nu)$ with $H= 1$ is a negative definite kernel for all $\mu, \nu$ in $\mathcal{W}_2(\Omega)$.\\
		We can easily see that $e^{-\lambda x}$ with $\lambda$ positive is a completely monotone function. Let us then consider a mapping as follows:
		\begin{align*}
		F: \mathbb{R}^{+} &\rightarrow \mathbb{R}^{+}\\
		{} x &\mapsto \gamma^2 e^{-\lambda x},
		\end{align*}
		where $\gamma^2 >0,\ x = W_2^{2}(\mu,\nu)$ with $\lambda = \frac{1}{l},\ l>0$. Then $F$ is also a completely monotone function. From Theorem \ref{scho}, $k_\Theta$ is a positive definite kernel.
	\end{proof}

%--------------------------------------------------------------------------------------------------------

\section{Regression}\label{sec3}
\subsection{Setting}\label{3.1}
	In this section, we aim to define a regression function with distribution inputs.  The problem of distribution regression consists in estimating an unknown function $f: \mathcal{W}_2(\Omega) \rightarrow \mathbb{R}$ by using observations  $(\mu_i, y_i)$ in $ \mathcal{W}_2(\Omega) \times \mathbb{R}$ for all $i=1,\cdots,n$ . We recall observes in (\ref{equa1}) as follows
	\begin{equation}\label{model1}
	y_i = f(\mu_i) + \epsilon_i.
	\end{equation} 
	To provide a general form for functions defined on distributions, we will use the RKHS framework. Let $k_\Theta:\ \mathcal{W}_2(\Omega)\times \mathcal{W}_2(\Omega) \mapsto \mathbb{R}$ be defined in Theorem \ref{exp.mono}. For a fixed valid $\Theta$, we define a space $\mathcal{F}_0$ as follows:
	\begin{equation*}
	\mathcal{F}_0:=  {\rm span} \left\{k_\Theta(\bullet ,\mu) : \mu \in \mathcal{W}_2(\Omega) \right\}.
	\end{equation*}	
	And $\mathcal{F}_0$ is endowed with the inner product 
	\begin{equation*}
	\langle f_n,g_m \rangle_{\mathcal{F}_0} = \sum_{i=1}^{n} \sum_{j=1}^{m}\alpha_i \beta_j k_\Theta(\mu_i,\nu_j),
	\end{equation*}
	where $f_n(\bullet)=\sum_{i=1}^{n}\alpha_ik_\Theta(\bullet,\mu_i)$ and $g_m(\bullet)=\sum_{j=1}^{m}\beta_jk_\Theta(\bullet,\nu_j)$. The norm in $\mathcal{F}_0$ corresponds to the inner product,
	\begin{equation}
	 \|f_n \|_{\mathcal{F}_0}^2= \sum_{i=1}^{n} \alpha_i^2k_\Theta(\mu_i,\mu_i).
	\end{equation}

	Let $\mathcal{F}$ be the space of all continuous real-valued functions from $\mathcal{W}_2(\Omega)$ to $\mathbb{R}$ while  $\mathcal{F}_0$ consists of all functions in $\mathcal{F}$ which are uniform limits of functions of form $f_n$. The kernel $k_\Theta$ is said to be universal if it has a property that $\mathcal{F}_0=\mathcal{F}$. \\
	\indent 	From that for all $f,g$ belong in $\mathcal{F}$, the inner product is well defined as following formula
	\begin{equation}
	\langle f,g \rangle_{\mathcal{F}}: = \lim_{n\rightarrow \infty} \langle f_n,g_n \rangle_{\mathcal{F}_0}.
	\end{equation}
	Coming back to our problem, we want to estimate the unknown function $f$ by an estimation function $\hat{f}$ obtained by minimizing the regularized empirical risk over the RKHS $\mathcal{F}$. For this consider, we solve the solution of the minimization problem
	\begin{equation}\label{arg}
	\hat{f} = \underset{f \in \mathcal{F}}{\text{argmin}}  \left( \sum_{i=1}^{n}|y_i-f(\mu_i)|^2 + \lambda \norm{f}_\mathcal{F}^2 \right),
	\end{equation} 
	where $\lambda \in \mathbb{R}^{+}$ is the regularization parameter. Using the Representer theorem, this  leads  to the following expression for $ \hat{f}$,
	\begin{equation}\label{f.F}
	\hat{f}: \: \mu \mapsto \hat{f}(\mu) := \sum_{j=1}^{n} \hat{\alpha}_j k_\Theta(\mu,\mu_j) ,
	\end{equation}
	where $\left\{\hat{\alpha}_j \right\}_{j=1}^n$ are parameters typically obtained from training data.
	
\subsection{Universal kernel}

First, we recall the definition of a  universal kernel  and the main theorem to ensure universal properties of positive definite kernels in Theorem \ref{exp.mono}.
	 
\begin{defi}\label{uni}
	Let $C(X)$ be the space of continuous bounded functions on compact domain $X$. A continuous kernel $k$ on domain $X$ is called universal if the space of all functions induced by $k$ is dense in $C(X)$, i.e, for all $ f \in C(X)$ and every $\epsilon >0$ there exists a function $g$ induced by $k$ with $\norm{f-g}_\infty \leq \epsilon$.
\end{defi}
For more information on universal kernel and RKHS, we refer to Chapter 4 in \cite{villani2008optimal} and \cite{sriperumbudur2011universality}, \cite{micchelli2006universal}.\\

Now, we introduce a main contribution of our work which shows that kernels based on Wassertein distance proposed in \cite{bachoc2017gaussian} are universal, finally, we get what kernel we need for the setting \ref{3.1}.
	
\begin{thm}\label{th.principal}
	Let choose the parameter $\Theta$ in (\ref{kernel}) such that $\gamma \neq0,\ l>0$ and $H=1$. The kernel $k_\Theta: \mathcal{W}_2(\Omega)\times \mathcal{W}_2(\Omega) \rightarrow \mathbb{R}$ defined in (\ref{kernel}) is universal.
\end{thm}
 To prove this Theorem, we need the two following Proposition \ref{quan} and Proposition \ref{printheo}. We will give the proof of this result after the two of following statements. 	
	
\begin{proposition}\label{quan}
		Let an increasing function $F_\mu|_\Omega : \Omega \rightarrow [a,b]$ with $[a,b]\subset [0,1]$ be the distribution function restricted on a compact subset $\Omega$ of $\mathbb{R}$, $F^{-1}_\mu|_\Omega$ be defined by $F_\mu^{-1}|_\Omega(t)= inf\{u\in \Omega,F_\mu|_\Omega(u) \geq t \}$ for all $t\in [a,b]$. Then $F_\mu|_\Omega$ is continuous if and only if $F^{-1}_\mu|_\Omega$ is strictly increasing on $[\text{inf}\ \text{ran} F_\mu|_\Omega, \text{sup}\ \text{ran} F_\mu|_\Omega]$. $F_\mu|_\Omega$ is strictly increasing if and only if ${F^{-1}_\mu|}_\Omega$ is continuous on $ran{F_\mu}|_\Omega$, where $ran F_\mu|_\Omega:= \left\{F_\mu|_\Omega (x): x\in \Omega \right\}$ denotes the range of $F_\mu|_\Omega$.
\end{proposition}
\begin{proof}
	For the proof of this proposition, we refer to \cite{embrechts2013note} for more theoretical details in the case of the quantile function from $\left[0,1\right]$ to $\mathbb{R}$. Here, we give the proof for the kind of function $F^{-1}_\mu|_\Omega$ from $\left[a,b\right] \in \left[0,1\right]$ to $\Omega$ be a compact subset of $\mathbb{R}$.
	
	First, we prove that $F_\mu|_\Omega$ is continuous $\Leftrightarrow\ $ $F_\mu^{-1}|_\Omega$ is strictly increasing on $[\text{inf}\ \text{ran} F_\mu|_\Omega, \text{sup}\ \text{ran} F_\mu|_\Omega]$.\\
	$\Rightarrow$) Assume that $F_\mu|_\Omega$ is discontinuous at $x_0 \in \Omega$, we need to prove that $F_\mu^{-1}|_\Omega$ is not strictly increasing on$[\text{inf}\ \text{ran} F_\mu|_\Omega, \text{sup}\ \text{ran} F_\mu|_\Omega]$.\\
	In fact, since $F_\mu|_\Omega$ is increasing, this implies that there exist
	\begin{align*}
	z_1 &:= F_\mu|_\Omega\left(x_0-\right):= \lim_{x \uparrow x_0} F_\mu|_\Omega(x)\\
	z_2 &:= F_\mu|_\Omega\left(x_0+\right):= \lim_{x \downarrow x_0} F_\mu|_\Omega(x),
	\end{align*}
	$z_1<z_2$ and $z_0:= F_\mu|_\Omega(x_0) \in \left[z_1,z_2\right]$ such that $z_1<z_0<z_2$.\\
	Now, there exists $z_3,z_4\in [a,b]$ such that either $z_1<z_3<z_4<z_0<z_2$ or $z_1<z_0<z_3<z_4<z_2$ and note that for all $z\in \left[z_3,z_4\right] \in [\text{inf}\ \text{ran} F_\mu|_\Omega, \text{sup}\ \text{ran} F_\mu|_\Omega]$, $z \not\in\text{ran} F_\mu|_\Omega $.\\
	From that, by definition of $F_\mu^{-1}|_\Omega$, it is a constant in $\left[z_3,z_4\right]$. Hence, $F_\mu^{-1}|_\Omega$ is not strictly increasing on $[\text{inf}\ \text{ran} F_\mu|_\Omega, \text{sup}\ \text{ran} F_\mu|_\Omega]$.\\
	$\Rightarrow$) Assume that $F_\mu^{-1}|_\Omega$ be not strictly increasing on $[\text{inf}\ \text{ran} F_\mu|_\Omega, \text{sup}\ \text{ran} F_\mu|_\Omega]$, we need to prove that $F_\mu|_\Omega$ is discontinuous.\\
	In fact, because $F_\mu^{-1}|_\Omega$ is not strictly increasing on $[\text{inf}\ \text{ran} F_\mu|_\Omega, \text{sup}\ \text{ran} F_\mu|_\Omega]$, there exists $\left[z_1,z_2\right]\in \Omega$ with $\text{inf}\ \text{ran} F_\mu|_\Omega<z_1<z_2<\text{sup}\ \text{ran} F_\mu|_\Omega$ such that $F_\mu^{-1}|_\Omega(z)=x,\ \forall z\in \left[z_1,z_2\right]$ and a $x\in \Omega$.\\
	By definition of $F_\mu^{-1}|_\Omega$, this implies that
	\begin{equation*}
	F_\mu|_\Omega \left(x-\epsilon\right) <z_1<z_2<F_\mu|_\Omega \left(x+\epsilon\right),
	\end{equation*}
	for all $\epsilon>0$.\\
	Letting $\epsilon \downarrow 0$, we obtain
	\begin{equation*}
	F_\mu|_\Omega (x-) \leq z_1<z_2\leq F_\mu|_\Omega \left(x+\right).
	\end{equation*}
	Thus $F_\mu|_\Omega$ is discontinuos at $x$.
	
	Now, we prove the second statement: $F_\mu|_\Omega$ is not strictly increasing $\Leftrightarrow \ $ $F_\mu^{-1}|_\Omega$ is discontinuous on $\text{ran} F_\mu|_\Omega$.\\

	To prove this statement, we need to use a result (7) in Proposition 2.3 in \cite{embrechts2013note}, we note this result by (*). Here, we recall briefly this result corresponding to our case as follows.
	\begin{equation*}
	\left(F_\mu^{-1}|_\Omega(z-),F_\mu^{-1}|_\Omega(z+) \right) \subseteq \left\lbrace x\in \Omega : F_\mu|_\Omega(x)=z \right\rbrace \subseteq \left[F_\mu^{-1}|_\Omega(z-),F_\mu^{-1}|_\Omega(z+)  \right],
	\end{equation*}
	where 
	\begin{align*}
	F_\mu^{-1}|_\Omega(z-) &= \lim_{t\uparrow z}F_\mu^{-1}|_\Omega(t),\\
	F_\mu^{-1}|_\Omega(z+) &= \lim_{t\downarrow z}F_\mu^{-1}|_\Omega(t).
	\end{align*}
	$\Rightarrow$) Now, assume that $F_\mu|_\Omega$ is not strictly increasing, we prove that $F_\mu^{-1}|_\Omega$ is discontinuous on $\text{ran} F_\mu|_\Omega$.\\
	In fact, let $A_z:= \left\lbrace x\in \Omega:\ F_\mu|_\Omega(x)=z \right\rbrace $.\\
	Because $F_\mu|_\Omega$ is not strictly increasing, there exists a $x\in \Omega$ such that $A_z$ contains an open interval.\\
	From the second statement of above result (*), we have 
	\begin{equation*}
	F_\mu^{-1}|_\Omega(z-) < F_\mu^{-1}|_\Omega(z+).
	\end{equation*}
	Thus $F_\mu^{-1}|_\Omega$ is discontinuous at $z \in \text{ran} F_\mu|_\Omega$.\\
	$\Leftarrow$) Assume $F_\mu^{-1}|_\Omega$ is discontinuous on $\text{ran} F_\mu|_\Omega$, we prove that $F_\mu|_\Omega$ is not strictly increasing.\\
	In fact, because $F_\mu^{-1}|_\Omega$ is discontinuous on $\text{ran} F_\mu|_\Omega$, there exists $z\in \text{ran} F_\mu|_\Omega$ such that $F_\mu^{-1}|_\Omega(z-) < F_\mu^{-1}|_\Omega(z+)$. Using the second statement in the result (*) that $A_z$ contains the open interval $\left(F_\mu^{-1}|_\Omega(z-),F_\mu^{-1}|_\Omega(z+) \right) $ thus $F_\mu|_\Omega$ is not strictly increasing.
\end{proof}

\begin{proposition}\label{printheo}
		Let $X$ be a compact metric space and $\mathcal{H}$ be a separable Hilbert space such that there exist a continuous and injective map: $\rho:\ X \rightarrow \mathcal{H}$. For $\gamma >0$, the Gaussian-type RBF-kernel $k_\gamma:\ X\times X\rightarrow \mathbb{R}$ is the universal kernel, where
		\begin{equation*}
		k_\sigma(x,x^{'}):= \exp(-\sigma^2\| \rho(x)-\rho(x^{'})\|^2_\mathcal{H}),\quad x,x^{'} \in X.
		\end{equation*}
\end{proposition}
\begin{proof}
	See a part iii) of Theorem 2.2 in \cite{christmann2010universal} for the proof.
\end{proof}
Now, using results of the two above results from Proposition \ref{quan} and \ref{printheo}, we give a proof for Theorem  \ref{th.principal}.
\begin{proof}
		From Proposition \ref{quan} with the conditions including the distribution restricted on $\Omega$, $F_\mu|_\Omega$ be continuous  and $F_\mu|_\Omega$ be strictly increasing on $\Omega$, then there exists a continuous and injective map 
		\begin{align*}
		\rho: \mathcal{W}_2(\Omega) &\rightarrow \mathbb{L}_2[a,b]\\
		{} \mu &\mapsto \rho(\mu):=F^{-1}_\mu|_\Omega.
		\end{align*}
		$F^{-1}_\mu|_\Omega$ is continuous on $[a,b]$ and strictly increasing on $[\text{inf}\ \text{ran} F_\mu|_\Omega, \text{sup}\ \text{ran} F_\mu|_\Omega]$.\\
		We consider a Wasserstein space $\mathcal{W}_2(\Omega)$ metrized by the Wasserstein distance $W_2$ with $\Omega$ be a compact subset on $\mathbb{R}$  and $\mathbb{L}_2[a,b]$ be the usual space of square integrable functions on $[a,b]$. For $\sigma$ in Proposition \ref{printheo} is exactly defined by $1/\sqrt{l}$ for all $l>0$. We have
		\begin{equation*}
		k_\Theta(\mu,\nu) = \gamma^2 \exp\left\{-\frac{\|F^{-1}_\mu|_\Omega - F^{-1}_\nu|_\Omega\|^2_{\mathbb{L}_2[a,b]}}{l}\right\}
		\end{equation*}
		is the universal kernel from Proposition \ref{printheo}. We complete the proof of Theorem \ref{th.principal}.
\end{proof}
The minimization program in \eqref{arg} can be solved explicitely using the representer theorem of \cite{kimeldorf1971some}.  Note that Sch\"olkopf and Smola \cite{smola1998learning} give a simple proof of a more general version of the theorem.  Define $c_{ij}$ as follows
	\begin{equation*}
	c_{ij} = \gamma^2\exp \left(-\frac{W^2_2(\mu_i,\mu_j)}{l} \right)
	\end{equation*}
	and $\alpha = \left(\alpha_1,\cdots,\alpha_n \right)^T,\ Y=(y_1,\cdots,y_n)^T$.\\
	Now we take the matrix formulation of \eqref{arg} we obtain
	\begin{equation}\label{matrix.form}
	\underset{\alpha}{\text{min}}\ trace((Y-C\alpha)(Y-C\alpha)^{T}) + \lambda trace(C\alpha \alpha^T),
	\end{equation}
	where the operation trace is defined as
	\begin{equation*}
	trace (A) = \sum_{i=1}^{n}a_{ii}
	\end{equation*}
	with $A=(a_{ii})_{i=1}^n$.\\
	Taking the derivative of (\ref{matrix.form}) with respect to vector $\alpha$, we find that $\alpha$ satisfies the system of linear equations 
	\begin{equation}\label{linear}
	(C+\lambda I)\alpha = Y.
	\end{equation}
	Hence 
\begin{equation}\label{formest}
\hat{f}(\mu) = \sum_{j=1}^n \hat{\alpha}_j k_\Theta(\mu,\mu_j),
\end{equation}
with
\begin{equation}\label{lambda2}
\hat{\alpha}= (C+\lambda I)^{-1}Y.
\end{equation}

%\newpage
\section{Numerical Simulations and Real data application}\label{sec4}
First we consider the regression model with one-dimensional distribution input. For this, we  train the model and try to estimate the unknown function $f$ using our proposed methodology. Then we address the regression model but with higher-dimensional distributions entries. By using the same technique for generating data in the one-dimensional distribution, we estimate the unknown function $f$ in the case of 2-dimensional distributions. Finally, we apply our for a   real dataset from biological field.
\subsection{Simulation in the regression model with one-dimensional distribution input}
\subsubsection{Overview of the simulation procedure}
In this section, we investigate the regression model for predicting the regression function from  distributions. Particularly, we want to estimate the unknown function $f$ in model (\ref{model1}) by using the proposed estimation $\hat{f}$ in (\ref{formest}), so we need to present how we can optimize the parameters in this formula. We then compare the regression model based on RKHS induced by our universal kernel function to more classical kernel functions operating on projections of the probability measures on finite dimensional spaces. 
We address the input-output map given by
\begin{equation}\label{exactf}
f(\nu) = \frac{m_\nu}{0.05+\sigma_\nu},
\end{equation}
where $\nu$ is a Gaussian distribution of mean $m_\nu$ and variance $\sigma^2_\nu$. We consider the ground truth function $f$ that we compare with a predicted function $\hat{f}$, such as:
\begin{equation}
\hat{f}(\nu)=\gamma^2 \sum_{j=1}^{n}\hat{\alpha}_j \exp \left[-\frac{W^2_2(\nu,\mu_j)}{l}\right],
\end{equation}
where the Wasserstein distance between two Gaussian distribution is calculated using:
\begin{align*}
W_2^2(\mu,\nu) &= (m_\mu - m_\nu)^2 + (\sigma_\mu - \sigma_\nu)^2,
\end{align*}
where $\mu \sim \mathcal{N}(m_\mu,\sigma^2_\mu)$ and $\nu\sim \mathcal{N}(m_\nu,\sigma^2_\nu)$.\\
Each value $\hat{\alpha}_j$ is estimated using Eq.~\eqref{lambda2} which depends on parameter $\lambda>0$. Thus our proposed estimation function $\hat{f}$ depends totally on the three parameters $\lambda >0, \gamma \neq 0$ and $l>0$. To understand the effects of these parameters on $\hat{f}$, we define reference values of $\lambda, \gamma$ and $l$.
We then generate a training set including the normal distributions $\nu_i \sim \mathcal{N}(m_{\nu_i},\sigma_{\nu_i})$ such that $cor(\nu_i,\nu_j) \neq 0, \forall i,j = 1,\cdots n$, with $n$ be a size of training set. In this simulation, we take $n=200$.\\
From the training set $\{\left(\nu_i,f(\nu_i)\right)\}_{i=1}^{n=200}$, we fit three regression models which we call "Wasserstein", "Legendre" and "Histogram". "Wassertein" stands  for our model, "Legendre" corresponds to regression models with smooth density input obtained by projecting onto Legendre polynomial basis while the last "Histogram" is introduced for the rough histogram input, as detailed  later. Then we evaluate the quality of the two regression models on a test set of size $n_t$ of the form $\{(\nu_{t,i},f(\nu_{t,i}))\}_i^{n_t}$, where $\nu_{t,i}$ is generated in the same way as $\nu_i$ above. We consider the following quality criteria, that is the root mean square error (RMSE) to see the qualify of our regression model by
\begin{equation*}
RMSE^2(\hat{f},f) = \frac{1}{n_t} \sum_{i=1}^{n_t} \left[f(\nu_{t,i}) - \hat{f}(\nu_{t,i})\right]^2.
\end{equation*}

\subsubsection{Detail on the regression models}
We refer to our model as "Wasserstein" and introduce briefly "Legendre" and "Histogram" regression models. 

The "Wasserstein" model first propose the estimated function as follows
\begin{equation}
\hat{f}(\nu_{t,i}) = \gamma^2 \sum_{j=1}^{n}\hat{\alpha_j}\exp \left[-\frac{(m_{\nu_j} - m_{\nu_{t,i}})^2 + (\sigma_{\nu_j} - \sigma_{\nu_{t,i}})^2}{l} \right],
\end{equation}
where $\nu_{t,i}, i=1,\cdots,n_t$ belong to testing set with size $n_t$, and $\nu_j, j=1,\cdots,n$ belong to training set with size $n$. The estimated function $\hat{f}$ depends on three parameters $\gamma \neq 0,\ \lambda >0$ and $l>0$.

The "Legendre" model is based on kernel functions operating on finite dimensional linear projections of the distributions. For a Gaussian distribution $\mu\sim \mathcal{N}(m,\sigma^2)$ with density $f_\mu(t)=\frac{1}{\sqrt{2\pi}\sigma}\exp\left(-\frac{(t-m)^2}{2\sigma^2} \right)$ and support $[0,1]$, we compute for $i=0,\cdots, \theta -1$:
\begin{equation*}
a_i(\mu) = \int_{0}^{1}\frac{1}{\sqrt{2\pi}\sigma}\exp\left(-\frac{(t-m)^2}{2\sigma^2} \right) p_i(t)dt,
\end{equation*}
where $p_i$ is the $i$-th normalized Legendre polynomial, with $\int_{0}^{1} p_i^2(t)dt = 1$. The integer $\theta$ is called the order of the decomposition. Then $k_L$ operators on the vector $(a_0(\nu),\cdots,a_{\theta-1}(\nu)$ and is of the form
\begin{equation}
k_L(\nu_1,\nu_2) = \gamma^2\exp \left[-\sum_{i=0}^{\theta-1}\frac{|a_i(\nu_1)-a_i(\nu_2)|}{l_i} \right].
\end{equation}
Thus the estimated regression function $\hat{f}$ in this case is calculated by following function 
\begin{equation*}
\hat{f}(\nu_{t,i}) =\gamma^2 \sum_{j=1}^{n}\hat{\alpha}_j  \exp \left[-\sum_{i=0}^{\theta-1}\frac{|a_i(\nu_{t,i})-a_i(\nu_j)|}{l_i} \right].
\end{equation*}
We just consider the orders of the decomposition $5$ and $10$. We fix $l_i=l$ for all $i=1,\cdots,\theta -1$, this estimated function depends also on three parameters $\gamma \neq 0,\ \lambda>0$ and $l>0$.

The "Histogram" model is based on a Gaussian RBF kernel with $\chi_2$-distance between the histograms, see \cite{vedaldi2009multiple}. More particularly,
for $0< \zeta \leq 1$,

\begin{equation*}
K\left(\nu_1,\nu_2\right)=\exp \left(-\zeta \chi_2\left(h_1,h_2\right) \right),
\end{equation*}
where 
\begin{equation*}
\chi_2(h_1,h_2) = \sum_{r=1}^{M} \frac{\left(h_1(r)-h_2(r)\right)^2 }{h_1(r)+h_2(r)},
\end{equation*}
with $h_i$ is the histogram of the $i$th Gaussian distribution $\nu_i,\ ,\ i=\left\lbrace 1,2\right\rbrace $ and $M$ is the number of bins for all histograms $h_i$.\\
Thus, in this case, the estimated regression function is considered by the following expression
\begin{equation*}
\hat{f}\left(\nu_{t,i} \right) = \sum_{j=1}^{n} \hat{\alpha}_j \exp \left(-\zeta \sum_{r=1}^{M} \frac{\left(h_1(r)-h_2(r)\right)^2 }{h_1(r)+h_2(r)} \right).
\end{equation*}

\subsubsection{Result}
In the simulations, we observe the effects of parameters $\lambda>0,\ \gamma\neq 0$ and $l>0$ on RMSE between predicted function $\hat{f}(\nu_{i,t})$ and exact function $f(\nu_{i,t})$ through the testing set $\{(\nu_{i,t})\}_{i=1}^{n_t}$. We also take two sizes of testing set $n_t=500,n_t=700$ to see the changes of RMSE. We just show the detailed presentation about choosing the optimal parameters on the "Wasserstein" model.
\paragraph{\textbf{Case of testing set size $n_t=500$}}:
Now we consider RMSE in the case of $n_t=500$ under the different fixed values $\gamma=1/2, 1, 10$ and running $\lambda >0$ separated by 30 values from $0.005$ to $30$, $l >0$ separated by 25 values from $0.005$ to $20$. Let us see here the values of RMSE with the different cases of $\gamma$ in following Figure \ref{RMSEfixgam1.2}, \ref{RMSEfixgam1}, \ref{RMSEfixgam10}. 

\begin{center}
	\begin{figure}[!htb]
		\includegraphics[width=0.9\textwidth]{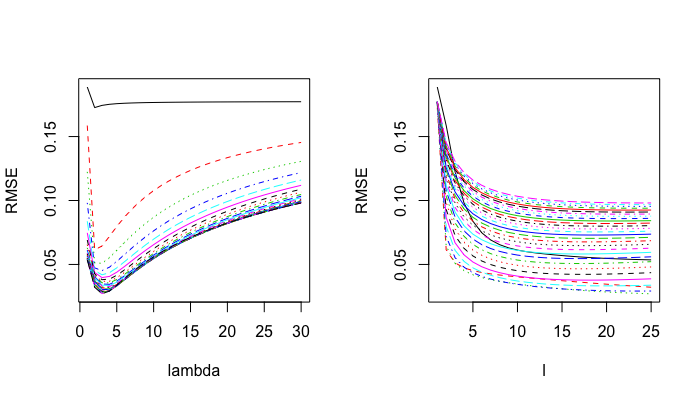}
		\caption{In the case of $n_t=500$, fixing a value $\gamma=1/2$, we run $\lambda >0$ separated by 30 values from $0.005$ to $30$, $l >0$ separated by 25 values from $0.005$ to $20$. We see the two graphs, one follows values of $\lambda$ in the left side and $l$ in the right side. RMSE will be minimized, in which it is lower than $0.08$, with $0<\lambda<15$ and $l$ big enough. We note that when $0<l<1$ RMSE is quite a big value for all $\lambda>0$, so we avoid to chose these values of $l$.}
		\label{RMSEfixgam1.2}
	\end{figure}
\end{center}

\begin{center}
	\begin{figure}[!htb]
		\includegraphics[width=0.9\textwidth]{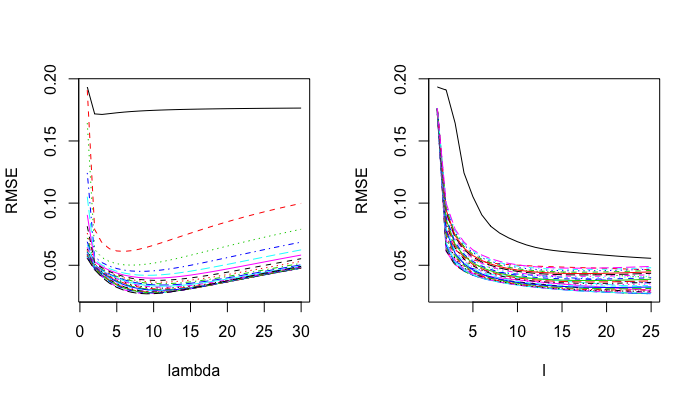}
		\caption{In the case of $n_t=500$, fixing a value $\gamma=1$, we run $\lambda >0$ separated by 30 values from $0.005$ to $30$, $l >0$ separated by 25 values from $0.005$ to $20$. We see the two graphs, one follows values of $\lambda$ in the left side and $l$ in the right side. The variations of RMSE in this case is not change significantly with the case of $\gamma=1/2$, however, it looks smaller than the case $\gamma=1/2$. We also see that RMSE will be minimized by two case: first $0<\lambda<1$ and $l$ big enough; second $\lambda >1$ for all $l>1$ and RMSE is quite big at $0<l<1$ for all $\lambda>0$.}
		\label{RMSEfixgam1}
	\end{figure}
\end{center}

\begin{center}
	\begin{figure}[!htb]
		\includegraphics[width=0.9\textwidth]{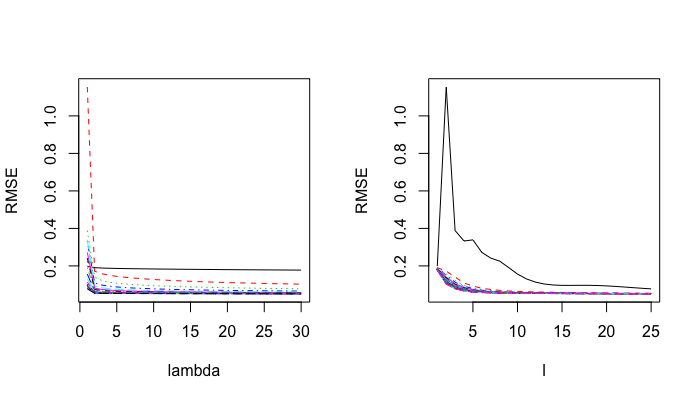}
		\caption{In the case of $n_t=500$, fixing a value $\gamma=10$, we run $\lambda >0$ separated by 30 values from $0.005$ to $30$, $l >0$ separated by 25 values from $0.005$ to $20$. RMSE in this case looks bigger than two above cases of $\gamma=1/2$ and $\gamma=1$.}
		\label{RMSEfixgam10}
	\end{figure}
\end{center}

We realize through three choices of $\gamma=1/2, 1, 10$ that the values $\gamma$ give the same impact of RMSE variations, but the smallest RMSE in the case $\gamma=1$. In following this stimulation, we fix the value of $\gamma=1$ and run the values of $\lambda$ and $l$ to see the changes of RMSE in the case of bigger size of testing set.

\paragraph{\textbf{Case of testing set size $n_t=700$}}:
Now we consider RMSE in the case of $n_t=700$ in Figure \ref{RMSEfixgam11size700} for a fixed value $\gamma=1$ and running $\lambda >0$ separated by 30 values from $0.005$ to $30$, $l >0$ separated by 25 values from $0.005$ to $20$. The parameters $\lambda$ and $l$ control the smoothness of the estimator. See the Figure \ref{optimalparameter}, to more clearly about our regression model with the exact function defined in (\ref{exactf}).

And finally, we consider the different RMSE's between "Wasserstein" and "Legendre" model by choosing the optimal values $\gamma=1$, $\lambda=5$ and $l=10$ under considering $n_t=700$. In Table \ref{tab1}, we show the values of RMSE quality criteria for the "Wasserstein" and "Legendre" distribution regression models. From the values of the RMSE criterion, the "Wasserstein" model clearly outperforms the other models. The RMSE of the "Legendre" models slightly decreases when the order increases, and stay well above the RMSE of the "Wasserstein" model.
\begin{center}
	\begin{figure}[!htb]
		\includegraphics[width=\textwidth]{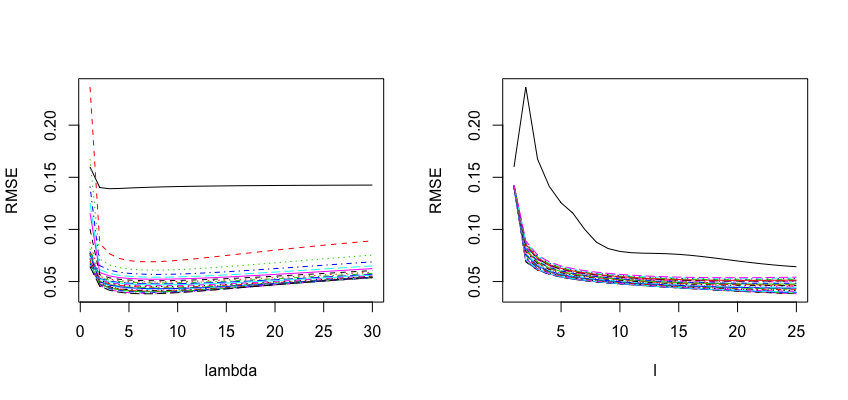}
		\caption{In the case of $n_t=700$, fixing a value $\gamma=1$, we run $\lambda >0$ separated by 30 values from $0.005$ to $30$, $l >0$ separated by 25 values from $0.005$ to $20$. RMSE is almost lower than $0.06$ when $\lambda >1$ and $l>1$. However for $0<\lambda<1$, we can also obtain the small RMSE when $l$ big enough. This figure provides a view about size of testing set, in which for the big enough of testing set size we will obtain the smaller RMSE under of the optimal parameters $\gamma,\ \lambda$ and $l$.}
		\label{RMSEfixgam11size700}
	\end{figure}
\end{center}

\begin{center}
	\begin{figure}[!htb]
		\includegraphics[width=\textwidth]{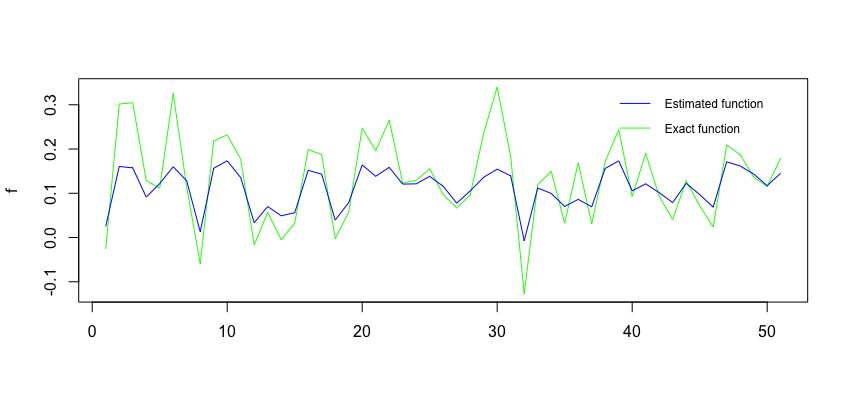}
		\caption{Regression function under exact and estimated function.
		The green line presents an exact function, which is many more variations, we desire to find the optimal parameters to obtain a more smooth curve. From Figure \ref{RMSEfixgam11size700}, we can chose these parameters following RMSE, however, in some cases it happens over-smoothing and under-smoothing. See in this Figure, the blue line looks like have some desirable properties when we chose the big enough values of $\lambda$ and $l$.}
		\label{optimalparameter}
	\end{figure}
\end{center}

\begin{table}[!htb]
	\begin{center}
		\begin{tabular}{||c | c||} 
			\hline
			model & RMSE \\ [0.5ex] 
			\hline 
			"Wasserstein" & $0.04$ \\[0.5ex] 
			\hline
			"Legendre" order 5 & $0.15$ \\[0.5ex] 
			\hline
			"Legendre" order 10 & $0.11$\\[0.5ex] 
			\hline
			"Histogram" & $0.12$\\[0.5ex] 
			\hline
		\end{tabular}
	\end{center}
	\caption{RMSE values of the quality criteria for the "Wasserstein", "Legendre" and "Histogram" distribution regression models. The "Wasserstein" is based on universal kernel function operating directly on the input Gaussian distributions, while "Legendre" is based on linear projections of the Gaussian distribution inputs on finite-dimensional spaces. For "Legendre", the order value is the dimension of the projection space. Beside that, the "Histogram" is producted by the rough histograms of Gaussian distributions. The quality criteria is the root mean square error (RMSE) should be minimal. The "Wasserstein" distribution regression model clearly outperforms the "Legendre" with two orders $5$, $10$ and the "Histogram". \label{tab1}} 
\end{table}
Hence from the Figure \ref{optimalparameter} and Table \ref{tab1}, we can see that by choosing the optimal parameters for $\gamma,\ l$ and $\lambda$ we can obtain a very well estimation function $\hat{f}$ without  under-smoothing or over-smoothing issues. Our regression model stay well above the RMSE criterion.

 Our interpretation for these results is that, because of the nature of the simulated data $(\mu_i,f(\mu_i))$ working directly on distributions and with the Wasserstein distance, is more appropriate than using linear projections. Indeed, in particular, two distributions with similar means and small variances are close to each other with respect to both Wasserstein distance and the value of the output function $f$. However,the probability density functions of the two distributions are very different from each other with respect to the $L^2$ distance in the case that the ratio between the two variances is large. Hence linear projections based on probability density functions is inappropriate in the setting considered here.
 
\subsection{Simulation in the regression model with two-dimensional distribution input}
In this section, we aim to study the two of estimated regression functions constructed by the Wasserstein distance and the Sliced Wasserstein distance in the case of 2D distribution entries. 
\subsubsection{Wasserstein distance and Sliced Wasserstein distance of 2D distributions}
For the Wasserstein distance between two 2-dimensional Gaussian distributions, we refer the construction of positive definite kernels by Hilbert space embedding of optimal transport maps in \cite{bachoc2018gaussian}. Firstly, we introduce briefly the notion of Wassertein barycenter. Let $\mathbb{P}$ be a distribution in $\mathcal{W}_2\left(\mathbb{W}_2\left(\mathbb{R}^p \right) \right)$, that is the set of measures on $\mathcal{W}_2\left(\mathbb{R}^p \right)$ with finite expected variances, and consider $\mu_1,\cdots,\mu_n$ i.i.d probabilities drawn according to the distribution $\mathbb{P}$. In this frame work, the Wasserstein distance between distribution on $\mathcal{W}_2\left(\mathbb{R}^p\right)$ is defined, for any $\nu\in \mathcal{W}_2\left(\mathbb{R}^p\right)$, as
\begin{equation*}
W^2_2\left(\mathbb{P},\delta_\nu\right) = \int W^2_2\left(\nu,\mu\right)d\mathbb{P}(\mu).
\end{equation*}
If $\tilde{\mu}$ is a random distribution obeying law $\mathbb{P}$, this corresponds to
\begin{equation*}
W_2^2\left(\mathbb{P},\delta_\nu \right) = \mathbb{E}_{\tilde{\mu} \sim \mathbb{P}} W^2_2\left(\tilde{\mu},\nu \right).
\end{equation*}
Note that we use the same notations for the Wassertein distances over distributions in $\mathcal{W}_2\left(\mathbb{R}^p\right)$ and over distributions on distributions in $\mathcal{W}_2\left(\mathbb{W}_2\left(\mathbb{R}^p \right) \right)$ because the space $\mathcal{W}\left(\mathbb{W}_2\left(\mathbb{R}^p \right) \right)$ inherits the properties of the space $\mathcal{W}_2\left(\mathbb{R}^p\right)$.

We define the Wassertein barycenter of $\mathbb{P}$ as a probability measure $\bar{\mu}$ in $\mathcal{W}_2\left(\mathbb{R}^p\right)$ such that
\begin{equation*}
\int W^2_2\left(\bar{\mu},\mu\right)d\mathbb{P}(\mu) = \inf \left\lbrace \int W^2_2\left(\nu,\mu\right)d\mathbb{P}(\mu), \nu \in \mathcal{W}_2\left(\mathbb{R}^p\right) \right\rbrace .
\end{equation*}
The following theorem from \cite{alvarez2015wide} guarantees the existence and uniqueness of this barycenter under some assumptions.
\begin{thm}
	Let $\mathbb{P}\in \mathcal{W}_2\left(\mathbb{W}_2\left(\mathbb{R}^p \right) \right)$. Assume that every distribution in the support of $\mathbb{P}$ is absolutely continuous with respect to Lebesgue measure on $\mathbb{R}^p$. Then there exists a unique distribution $\bar{\mu} \in \mathcal{W}_2\left(\mathbb{R}^p\right)$ defined as
	\begin{equation*}
	\bar{\mu} = \arg \min_{\nu \in \mathcal{W}_2\left(\mathbb{R}^p\right)} \left\lbrace \int W^2_2 \left(\nu,\mu \right)d\mathbb{P}(\mu)  \right\rbrace .
	\end{equation*}
\end{thm}

Considering a reference distribution $\bar{\mu} \in \mathcal{W}_2\left(\mathbb{R}^p\right)$, for $\mu \in \mathcal{W}_2\left(\mathbb{R}^p\right)$, let $T_\mu : \mathbb{R}^p \rightarrow \mathbb{R}^p$ we refer the optimal transportation maps defined by
\begin{equation*}
T_{\mu \sharp}\mu = \bar{\mu},
\end{equation*}
where $f_{\sharp}\pi= \pi \circ f^{-1}$ is the push-forward measure of a function $f$ from a measure $\pi$, and
\begin{equation*}
\norm{id - T_\mu}_{L^2(\mu)} = W_2\left(\mu, \bar{\mu} \right).
\end{equation*}
Note that the map $T_\mu$ is uniquely defined when $\mu$ is absolutely continuous w.r.t Lebesgue measure. Futhermore, $T_\mu$ is invertible from the support of $\mu$ to the support of $\bar{\mu}$ if also $\bar{\mu}$ is absolutely continuous. 

By associating the transport map $T_\mu^{-1}$ to each distribution $\mu$ and using positive definite kernel on the Hilbert $L^2\left(\bar{\mu} \right)$, containing these transport maps, the groups of authors in \cite{bachoc2018gaussian} introduce the following construction of positive definite kernels by Hilbert space embedding of optimal transport maps.
\begin{proposition}
	Consider a function $F: \mathbb{R}^{+} \rightarrow \mathbb{R}$ such that, for any Hilbert space $H$ with norm $\norm{.}_H$, the function $h_1,h_2 \rightarrow F\left(\norm{h_1-h_2}_H \right)$ is positive definite on $H$. Let $\bar{\mu}$ be a continuous distribution in $\mathcal{W}_2\left(\mathbb{R}^p\right)$. Consider the function $K$ on the set continuous distributions in $\mathcal{W}_2\left(\mathbb{R}^p\right)$ defined by
	\begin{equation} \label{ker2D}
	K(\mu,\nu)=F\left( \norm{T^{-1}_\mu - T^{-1}_\nu}_{L^2\left(\bar{\mu}\right)} \right).
	\end{equation}
	Then K is positive definite. 
\end{proposition}

Note that in the particular case of Gaussian distributions, the transport map is linear. Thus the kernel (\ref{ker2D}) based on the optimal transport map of Gaussian distributions is guaranteed universality. Using this result, we make later the link to our theory in building an estimated function for the unknown function $f$ in our model corresponding to two-dimensional Gaussian distribution entries. In the general dimension of Gaussian distributions, we refer to the section 4.3 in \cite{bachoc2018gaussian} for more details. 

We introduce now the kernel obtained by the Sliced Wasserstein distance, the one built in \cite{KolouriZouRohde, kolouri2018sliced}. The principal idea of this distance is to first obtain a family of one-dimensional probability distributions by using the Radon transforms. Then we calculate the distance between two multidimensional distributions as the Wasserstein distance of theirs one-dimensional representations. \\
Let $\nu$ and $\mu$ be two continuous probability measures on $\mathbb{R}^d$ with corresponding the positive probability density function $I_\nu$ and $I_\mu$, the Sliced Wasserstein distance is defined by
\begin{equation*}
K(\mu,\nu)=\exp \left(-\xi SW^2\left(\mu,\nu \right) \right),
\end{equation*}
where $0<\xi\leq 1$ and 
\begin{equation*}
SW^2\left(\mu,\nu \right) = \int_{\mathbb{S}^{d-1}}^{} W^2_2\left(\mathcal{R}I_\mu \left(.,\theta\right), \mathcal{R}I_\nu \left(.,\theta\right) \right ) d\theta ,
\end{equation*}
with $\mathcal{R}$ is the Radon transform and $\mathbb{S}^{d-1}$ is the unit sphere in $\mathbb{R}^d$.

\subsubsection{Detail on the simulation procedure}
We describe clearly our simulation in this section. We address the input-output map given by
\begin{equation*}
f(\nu)=\frac{\|m_\nu\|_2}{0.05+\|\Sigma_\nu\|_F},
\end{equation*}
where $\nu$ is a two-dimensional Gaussian distribution of mean $m_\nu$ and non-degenerated covariance $\Sigma_\nu$, $\|.\|_F$ is the Frobenious norm.\\
We consider the predicted function $\hat{f}$ as follows
\begin{equation*}
\hat{f}\left(\nu\right) =  \sum_{j=1}^{n}\hat{\alpha}_j K\left(\nu,\mu_j \right),
\end{equation*}
and we aim to compare the qualify of our estimated function $\hat{f}$ corresponding to the Wasserstein distance and the one corresponds to the Sliced Wasserstein distance. 

Particularly, the Wasserstein distance between two two-dimensional Gaussian distributions is defined by
\begin{equation*}
W_2^2\left(\mathcal{N}_2\left(m_1,\Sigma_1\right),\mathcal{N}_2\left(m_2,\Sigma_2\right)  \right) = \|m_1-m_2\|^2_2 + \|\Sigma_1^{1/2}-\Sigma_2^{1/2}\|^2_F.
\end{equation*}
From that, for $\mu = \mathcal{N}_2\left(m_\mu,\Sigma_\mu \right)$ and $\nu = \mathcal{N}_2\left(m_\nu,\Sigma_\nu \right)$ we consider the kernel $K$ such as
\begin{equation*}
K(\mu,\nu)=\gamma^2\exp \left(-\frac{\|m_\mu-m_\nu\|^2_2 + \|\Sigma_\mu^{1/2}-\Sigma_\nu^{1/2}\|^2_F}{l} \right),
\end{equation*}
where $\gamma \neq 0$ and $l>0$. We call this model by "Wasserstein2D".

For the Sliced Wasserstein distance, let $\mu=\mathcal{N}_2\left(m_\mu,\Sigma_\mu\right)$, a slice/projection of the Radon transform of $\mu$ is then a one-dimensional Gaussian distribution $\mathcal{R}\mu = \mathcal{N}\left(\theta m_\mu,\theta^T\Sigma_\mu \theta \right)$, with $\theta \in \mathbb{S}^1$. Then we use the Wasserstein distance for two of one-dimensional Gaussian distributions. We are interested in the kernel defined in \cite{KolouriZouRohde}, precisely
\begin{equation}
K(\mu,\nu)=\exp \left(-\xi \left(  \norm{\theta m_\nu-\theta m_\mu}_2^2 + \norm{\left(\theta^T\Sigma_\nu \theta\right)^{1/2} -\left(\theta^T\Sigma_\mu \theta\right)^{1/2 }}_F^2 \right)  \right),
\end{equation}
where $\xi>0$. We denote this model by "SlicedWasserstein2D".

Use the same technique in generating a testing and training set the one-dimensional Gaussian distribution cases. We consider the optimal estimates of both cases of "Wasserstein2D" and "SlicedWasserstein2D" regression models. More precisely, we optimize parameters $\lambda>0,\ \gamma \neq 0$ and $l>0$ in "Wasserstein2D" model and $\theta \in \mathbb{S}^1$ such as $\theta=\left(\cos \left(\frac{c\pi}{50} \right), \sin \left(\frac{c\pi}{50} \right) \right)$ with $c\in \left\lbrace 0,1,\cdots,50 \right\rbrace $, $\xi >0,\ \lambda >0$ for "SlicedWasserstein2D" model. We consider the RMSE criteria for these models following the level of noise $\sigma = 1,5,10$ in the Table \ref{tabb}. \\
\begin{table}[!htb]
	\begin{center}
		\begin{tabular}{||c|c | c||} 
			\hline
			$\sigma$&"Wasserstein2D" & "SlicedWasserstein2D" \\ [0.5ex] 
			\hline
			$\sigma=1$ &$0.06$&$0.06$\\ [0.5ex] 
			\hline
			$\sigma=5$ &$0.09$&$0.33$\\ [0.5ex] 
			\hline
			$\sigma=10$ &$0.1$&$0.91$\\ [0.5ex] 
			\hline
		\end{tabular}
	\end{center}
	\caption{RMSE values of the quality criteria for the "Wassertein2D" and "SlicedWassrstein2D" regression models following the level of noise $\sigma=1,5,10$ with two-dimensional distribution entries. Basing on the minimal RMSE values, the "Wasserstein2D" model clearly outperforms the "SlicedWassertein2D" model when there exists much noise in the regression model. \label{tabb}} 
\end{table}
Hence from the Table \ref{tabb}, we can see that by choosing the optimal parameters for $\gamma, l, \lambda$ and $\theta$ for obtaining the RMSE of both regression models, in the low noise level, the RMSE of both models are the same as quite good, but the "Wasserstein2D" model looks better than "SlicedWassertein2D" model when the noise is increased. 

\subsection{Application on evolution of hearing sensitivity}
An otoacoustic emission (OAE) is a sound which is generated from within the inner ear. OAEs can be measured with a sensitive microphone in the ear canal and provide a noninvasive measure of cochlear amplification (see Chapter: Hearing basics in \cite{eggermont2017hearing}). Recording of OAEs has become the main method for newborn and infant hearing screening (see Chapter: Early Diagnosis and Prevention of Hearing Loss in \cite{eggermont2017hearing}). There are two types of OAEs: spontaneous otoacoustic emissions (SOAEs), which can occur without external stimulation, and evoked otoacoustic emissions (EOAEs), which require an evoking stimulus. In this paper, we consider a type of EOAEs that is Transient-EOAE (TEOAE) (see for instance in \cite{joris2011frequency}), in which the evoked response from a click covers the frequency range up to around 4kHz. More precisely, each TEOAE models the ability of the cochlea to response to some frequencies in order to transform a sound into an information that will be processed by the brain.  So to each observation is associated a curve (the Oto-Emission curve) which describes the response of the cochlea at several frequencies to a sound. The level of response depends on each individual and each stimulus should be normalized, but the way each individual reacts is characteristic of its physiological characteristic. Hence to each individual is associated a curve, which after normalization, it is considered as a distribution $\mu$ describing the repartition of the responses for different frequencies ranging from 0 to 10 kHz. These distributions are shown in Figure~\ref{curveexample} and Table \ref{table:data}.

\begin{table}[h!]
	\centering
	\scalebox{0.9}{
\begin{tabular}{|p{2cm}|c| c|c| c|c|c|c|c|c|c|}
		\hline 
		\textbf{Name} & \textbf{Age} & \textbf{0(Hz)} & \textbf{39.06} & \textbf{78. 12}& $\cdots$ &\textbf{1171.88}&\textbf{1210.94}&$\cdots$& \textbf{9765.62} &\textbf{9804.69}\\
		\hline
		ABBAS & 23 & 0 & 0.0006 &0.0013 &$\cdots$ &0.0819&0.0388&$\cdots$ &0.0021&0.0015 \\
		ADAMS & 27 &0.0001 & 0.0010 &0.0022 &$\cdots$&0.0283&0.0283&$\cdots$ &0.0011&0.0006 \\
		ADENIYI & 30 & 0.0002 &0.0003 &0.0014&$\cdots$&0.0231&0.0065&$\cdots$&0.0012&0.0016 \\
		DUPLOOY &17&0.0003&0.0005&0.0015&$\cdots$&0.0786&0.1272&$\cdots$ &0.0036 &0.0031\\
		\vdots&\vdots&\vdots&\vdots&\vdots&\vdots&\vdots&\vdots&\vdots&\vdots&\vdots\\
		TRIMM&20&0.0005&0.0006&0.0026&$\cdots$ &0.0133&0.0215&$\cdots$ &0.0002&0.0017\\
		VELE&26&0.0001&0.0005&0.0018&$\cdots$ &0.1176&0.0859&$\cdots$ &0.0003&0.0005\\
		WALLER&40&0.0001&0.0001&0.0003&$\cdots$ &0.0178&0.0210&$\cdots$ &0.0014&0.0013\\
		WILLIAM &22&0.0002&0.0003&0.0009&$\cdots$ &0.0156&0.0656&$\cdots$ &0.0014&0.0018\\
		\hline
\end{tabular}}
%\captionsetup{justification=centering}
\caption{TEOAE data. 48 individuals are considered in human population, recorded in South Africa, with last names in first column, their exact ages in second column and the others describe the responses of the cochlea at several frequencies ranging from 0Hz to 10kHz.}
\label{table:data}
\end{table}

\begin{center}
	\begin{figure}
		\includegraphics[width=\textwidth]{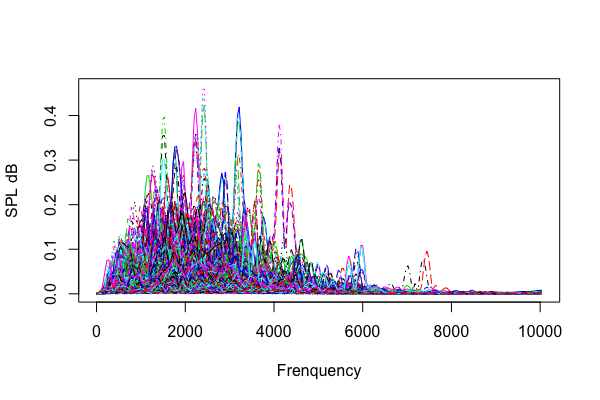}
		\caption{Oto-emission curves. 48 TEOAE curves following to frequencies ranging from 0Hz to 10kHz.}
		\label{curveexample}
	\end{figure}
\end{center}
%\newpage
The relationship between age and hearing sensitivity is investigated in \cite{o1994age, collet1990age} The results show that when age increases, the presence of EOAEs by age group and the frequency peak in spectral analysis decreases and EOAE threshold increases. The differences in EOAE have been also reported between age classes in humans. These results convey the idea that the response evolves with age and that the effect of ages in hearing issues is deeply related to  the changes of the cochlear properties. Hence our model uses as input these distributions $\mu$ and try to build a regression model to link between the age and these distributions representing the response of the cochlea at frequencies ranging from 0Hz to 10kHz. More precisely, we estimate the age for each level of response normalized and treated as a distribution $\mu$ by using our proposed function as follows
%Thus, there is an effect of age upon EOAEs. Following this idea for question that can we predict the age of subject with its given EOAE. 

%%In this section, we aim at studying the relationship between age and hearing sensitivity.  If an extensive medical literature exists on hearing loss,  little is known about the fundamental evolution of hearing capacity. For this we consider Oto-acoustics emissions, called transient oto-acoustic emissions also referred to as  TEOAE which  are most used in clinical examination. OAE are recorded by a small probe place in the external ear canal measure how the cochlea responds to a stimulus. Transiently Evoked Oto-Acoustic-Emission (TEOAE), commonly elicited by the use of a brief acoustic stimulus, used by ENT specialists as hearing loss diagnostic tool, are by-product of the tuned mechanical amplification responsible for the ear's impressive sensitivity (see for instance in \cite{joris2011frequency} and references therein). They reflect the objective and repeatable hearing sensibility of each individual. \\
%

\begin{equation}\label{hatfest}
	\hat{f}(\mu_i) = \gamma^2 \sum_{j=1}^{n}\hat{\alpha}_j \exp \left( - \frac{\int_0^1 (F^{-1}_{\mu_i} (t) - F^{-1}_{\mu_j} (t))^2 dt}{l}\right),
\end{equation}
where $F^{-1}_\mu (U)$ defined as (\ref{quant2}) and the value of $\hat{\alpha}_j$ is chosen by optimal parameter $\lambda$ in (\ref{lambda2}). We estimate the integral in (\ref{hatfest}) by following formula 
\begin{equation}\label{experimentdist}
\int_0^1 (F^{-1}_{\mu_i} (t) - F^{-1}_{\mu_j} (t))^2 dt =\sum_{m=1}^{M}\left[F^{-1}_{\mu_i} \left(\frac{m}{M}\right) - F^{-1}_{\mu_j} \left(\frac{m}{M}\right)\right]^2,
\end{equation}
where we can understand each $F_{\mu_i}$ is an experimental distribution function of $\mu_i$ and $M$ is the number of discretized frequencies. As far as we know, each individual is associated with a curve, which after normalization without lost relationship among original data, it is considered as a distribution $\mu_i$. To calculate $F^{-1}_{\mu_i} \left(\frac{m}{M}\right)$, we arrange each curve in ascending order, for instance we denote $X_\mu (1)\leq X_\mu (2)\leq \cdots \leq X_\mu (M)$  following to distribution $\mu$ and $\sum_{m=1}^{M}X_\mu(m) = 1$ , so $F^{-1}_{\mu_i} \left(\frac{m}{M}\right) = X_{\mu_i}(m)$. Hence, we write again the formula (\ref{experimentdist}) 
\begin{equation}
\int_0^1 (F^{-1}_{\mu_i} (t) - F^{-1}_{\mu_j} (t))^2 dt  = \sum_{m=1}^{M}\left(X_{\mu_i}(m) - X_{\mu_j}(m)\right)^2,
\end{equation}
where $X_{\mu_i}$ is a curve $X$ following to distribution $\mu_i$.\\
In our simulation, we choose $\gamma=1$, the value of $l =10$  and the value of $\lambda>0$. We aim to study the age in relation with its TEOAE curve of 48 subjects, recorded on human population in South Africa, with the range of frequency from 0Hz to 10kHz. See the Figure \ref{age} to show the differences between the age of 15 to 50 years old. Following the estimated function in (\ref{hatfest}), we take 47 distributions $\{\mu_j\}_{j=1}^{47}$ for training set to calculate estimation value of $\hat{\alpha}_j$ and try to estimate real age of a remaining individual $\mu_i$ with $i\neq j$.  And the results are showed clearly in the Figure \ref{realetpredict} and Figure \ref{fig:realandpredictedage} about the exact age and predicted age.

\begin{center}
	\begin{figure}
		\includegraphics[width=\textwidth]{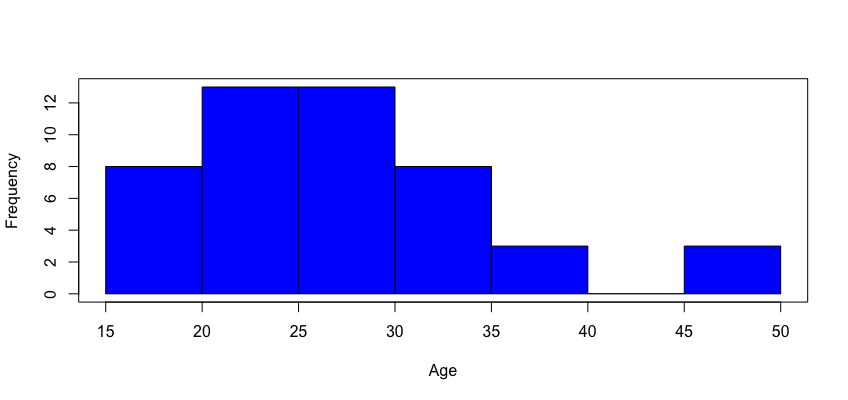}
		\caption{Histogram of real age in a human population. The age distribute diversity from 15 to 30, however, there is a few of individual of age from 35 to 40 and 45 to 50. And there exists no individual have age from 40 to 45.}
		\label{age}
	\end{figure}
\end{center}

\begin{center}
	\begin{figure}
		\includegraphics[width=\textwidth]{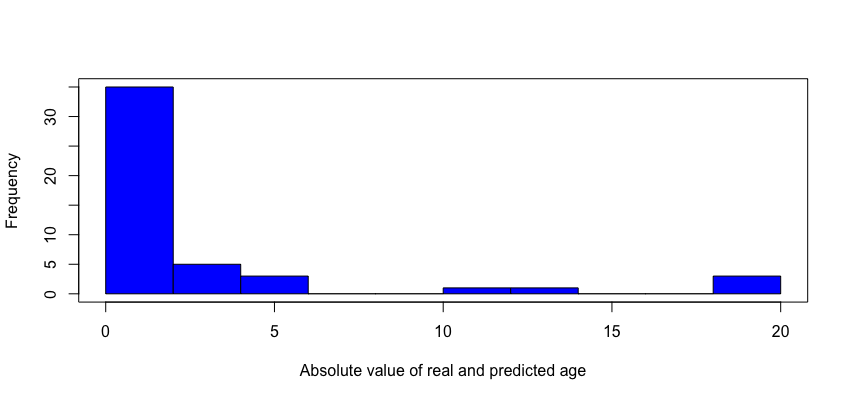}
		\caption{Histogram of difference between real and predicted age for OAE. In the first column, in which the difference between real and estimate age is very small closing to zero, this means more accuracy between real and predicted age. Almost the ages from 20 to 35 lie in this column.}
		\label{realetpredict}
	\end{figure}
\end{center}

\begin{center}
	\begin{figure}
		\includegraphics[width=\textwidth]{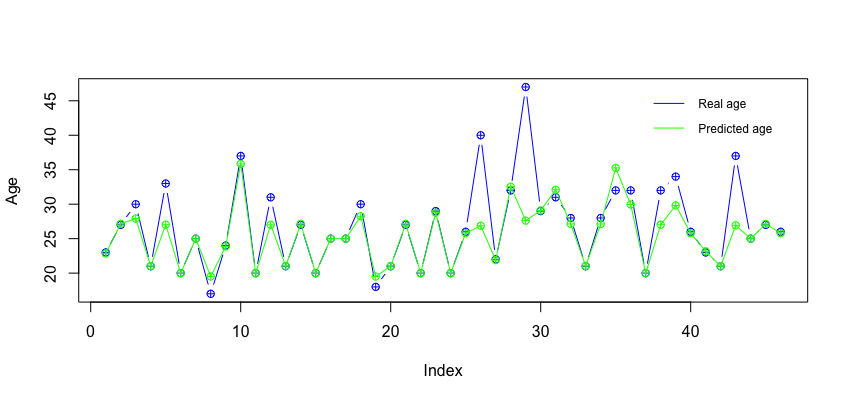}
		\caption{Real and Predicted Age. By using the optimized parameters of $\gamma=1$, $l=10$ and $\lambda >0$ depending on age class, we obtain almost the exacted ages belonging to the age class $[20,30]$ with $\lambda$ around 15. For instance, we can predict very well the exact ages $20,\ 21,\ 23,\ 24,\ 25,\ 27,\ 29$ corresponding to the predicted ages $19.50,\ 21.02,\ 22.83,\ 23.87,\ 24.89,\ 27.15,\ 28.76$.}
		\label{fig:realandpredictedage}
	\end{figure}
\end{center}

Hence in figure \ref{age} and Figure \ref{fig:realandpredictedage}, we applied effectively our proposed estimation function in predicting age from its TEOAE data. By choosing the optimal parameters $\gamma,\ l$ and $\lambda$ we could predict very well the exact ages belonging to the age class $[20,30]$ and negligible errors in other age classes. This is quite reasonable when seeing in the Figure \ref{age} that the age distributed diversity almost from 20 to 30 years old, so our proposed estimation function learnt very well to predict age in this age class. Thus by using the distribution regression model, we investigated the relationship between the evoked responses from clicks covering the frequencies range up to 10kHz and its evolutionary ages. 
%	\pagebreak

\section{Discussion}\label{sec5} 
In this paper, we have introduced a new estimated function for regression model with distribution inputs. More precisely, we effectively used class of positive definite kernel produced by Wasserstein distance, built in \cite{bachoc2017gaussian} by proving that it is a kind of universal kernel. Researching the  universal kernel theories, we detected a very good property of our universal kernel to build a RKHS. Then we obtained a particular estimation from Representer theorem for our distribution regression problem, these works showed that the relation between the random distribution and the real number response can be learnt by using directly the regularized empirical risk over RKHS. Our proposed estimation is clearly better than state-of-the-art-ones in simulated data. More interestingly, we researched successfully TEOAE curve of each individual in human population as a distribution which after normalization. We then investigated the relationship between age and its TEOAE that the response involves with age and the effect of age in hearing issues is deeply related to the change of cochlear. This is a new interesting approach in the field of Biostatistics, in which we indicated the evolution of hearing capacity under statistical domain - distribution regression model. We believe that our paper tackles an important issue for data science experts willing to predict problems in regression with probability distributions as input. The extension of this work on distributions for general dimensions should be addressed in a further work, using for instance as a kernel the one built in \cite{bachoc2018gaussian}.

%--------------------------------------------------------------------------------------------------------

	\pagebreak

%--------------------------------------------------------------------------------------------------------
%\newpage
\nocite{*}
\bibliographystyle{apalike} 
\bibliography{WRKHS.bib}

\begin{thebibliography}{}

\bibitem[Alvarez-Esteban et~al., 2015]{alvarez2015wide}
Alvarez-Esteban, P.~C., del Barrio, E., Cuesta-Albertos, J., and Matr{\'a}n, C.
  (2015).
\newblock Wide consensus for parallelized inference.
\newblock {\em arXiv preprint arXiv:1511.05350}.

\bibitem[Aza{\"\i}s and Bardet, 2012]{azais2006modele}
Aza{\"\i}s, J.-M. and Bardet, J.-M. (2012).
\newblock {\em Le mod{\`e}le lin{\'e}aire par l'exemple: r{\'e}gression,
  analyse de la variance et plans d'exp{\'e}rience illustr{\'e}s par R et SAS}.
\newblock Dunod.

\bibitem[Bachoc et~al., 2017]{bachoc2017gaussian}
Bachoc, F., Gamboa, F., Loubes, J.-M., and Venet, N. (2017).
\newblock A gaussian process regression model for distribution inputs.
\newblock {\em IEEE Transactions on Information Theory}, 64(10):6620--6637.

\bibitem[Bachoc et~al., 2018]{bachoc2018gaussian}
Bachoc, F., Suvorikova, A., Loubes, J.-M., and Spokoiny, V. (2018).
\newblock Gaussian process forecast with multidimensional distributional
  entries.
\newblock {\em arXiv preprint arXiv:1805.00753}.

\bibitem[Berlinet and Thomas-Agnan, 2011]{berlinet2011reproducing}
Berlinet, A. and Thomas-Agnan, C. (2011).
\newblock {\em Reproducing kernel Hilbert spaces in probability and
  statistics}.
\newblock Springer Science \& Business Media.

\bibitem[Christmann and Steinwart, 2010]{christmann2010universal}
Christmann, A. and Steinwart, I. (2010).
\newblock Universal kernels on non-standard input spaces.
\newblock In {\em Advances in neural information processing systems}, pages
  406--414.

\bibitem[Collet et~al., 1990]{collet1990age}
Collet, L., Moulin, A., Gartner, M., and Morgon, A. (1990).
\newblock Age-related changes in evoked otoacoustic emissions.
\newblock {\em Annals of Otology, Rhinology \& Laryngology}, 99(12):993--997.

\bibitem[Cowling, 1983]{cowling1983harmonic}
Cowling, M.~G. (1983).
\newblock Harmonic analysis on semigroups.
\newblock {\em Annals of Mathematics}, pages 267--283.

\bibitem[Eggermont, 2017]{eggermont2017hearing}
Eggermont, J.~J. (2017).
\newblock {\em Hearing Loss: Causes, Prevention, and Treatment}.
\newblock Academic Press.

\bibitem[Embrechts and Hofert, 2013]{embrechts2013note}
Embrechts, P. and Hofert, M. (2013).
\newblock A note on generalized inverses.
\newblock {\em Mathematical Methods of Operations Research}, 77(3):423--432.

\bibitem[Gretton et~al., 2012]{gretton2012kernel}
Gretton, A., Borgwardt, K.~M., Rasch, M.~J., Sch{\"o}lkopf, B., and Smola, A.
  (2012).
\newblock A kernel two-sample test.
\newblock {\em Journal of Machine Learning Research}, 13(Mar):723--773.

\bibitem[Joris et~al., 2011]{joris2011frequency}
Joris, P.~X., Bergevin, C., Kalluri, R., Mc~Laughlin, M., Michelet, P., van~der
  Heijden, M., and Shera, C.~A. (2011).
\newblock Frequency selectivity in old-world monkeys corroborates sharp
  cochlear tuning in humans.
\newblock {\em Proceedings of the National Academy of Sciences},
  108(42):17516--17520.

\bibitem[Kadri et~al., 2010]{kadri2010nonlinear}
Kadri, H., Duflos, E., Preux, P., Canu, S., and Davy, M. (2010).
\newblock Nonlinear functional regression: a functional rkhs approach.
\newblock In {\em Thirteenth International Conference on Artificial
  Intelligence and Statistics (AISTATS'10)}, volume~9, pages 374--380.

\bibitem[Kimeldorf and Wahba, 1971]{kimeldorf1971some}
Kimeldorf, G. and Wahba, G. (1971).
\newblock Some results on tchebycheffian spline functions.
\newblock {\em Journal of mathematical analysis and applications},
  33(1):82--95.

\bibitem[Kolouri et~al., 2018]{kolouri2018sliced}
Kolouri, S., Rohde, G.~K., and Hoffmann, H. (2018).
\newblock Sliced wasserstein distance for learning gaussian mixture models.
\newblock In {\em Proceedings of the IEEE Conference on Computer Vision and
  Pattern Recognition}, pages 3427--3436.

\bibitem[Kolouri et~al., 2015]{KolouriZouRohde}
Kolouri, S., Zou, Y., and Rohde, G.~K. (2015).
\newblock Sliced {W}asserstein kernels for probability distributions.
\newblock {\em CoRR}, abs/1511.03198.

\bibitem[Kutner et~al., 2004]{kutner2004applied}
Kutner, M.~H., Nachtsheim, C., and Neter, J. (2004).
\newblock {\em Applied linear regression models}.
\newblock McGraw-Hill/Irwin.

\bibitem[Micchelli et~al., 2006]{micchelli2006universal}
Micchelli, C.~A., Xu, Y., and Zhang, H. (2006).
\newblock Universal kernels.
\newblock {\em Journal of Machine Learning Research}, 7(Dec):2651--2667.

\bibitem[Muandet et~al., 2012]{muandet2012learning}
Muandet, K., Fukumizu, K., Dinuzzo, F., and Sch{\"o}lkopf, B. (2012).
\newblock Learning from distributions via support measure machines.
\newblock In {\em Advances in neural information processing systems}, pages
  10--18.

\bibitem[Muandet et~al., 2017]{muandet2017kernel}
Muandet, K., Fukumizu, K., Sriperumbudur, B., Sch{\"o}lkopf, B., et~al. (2017).
\newblock Kernel mean embedding of distributions: A review and beyond.
\newblock {\em Foundations and Trends{\textregistered} in Machine Learning},
  10(1-2):1--141.

\bibitem[Neter et~al., 1996]{neter1996applied}
Neter, J., Kutner, M.~H., Nachtsheim, C.~J., and Wasserman, W. (1996).
\newblock {\em Applied linear statistical models}, volume~4.
\newblock Irwin Chicago.

\bibitem[O-Uchi et~al., 1994]{o1994age}
O-Uchi, T., Kanzaki, J., Satoh, Y., Yoshihara, S., Ogata, A., Inoue, Y., and
  Mashino, H. (1994).
\newblock Age-related changes in evoked otoacoustic emission in normal-hearing
  ears.
\newblock {\em Acta Oto-Laryngologica}, 114(sup514):89--94.

\bibitem[Oliva et~al., 2014]{oliva2014fast}
Oliva, J., Neiswanger, W., P{\'o}czos, B., Schneider, J., and Xing, E. (2014).
\newblock Fast distribution to real regression.
\newblock In {\em Artificial Intelligence and Statistics}, pages 706--714.

\bibitem[Peyr{\'e} et~al., 2016]{peyre2016gromov}
Peyr{\'e}, G., Cuturi, M., and Solomon, J. (2016).
\newblock Gromov-{W}asserstein averaging of kernel and distance matrices.
\newblock In {\em ICML 2016}.

\bibitem[P{\'o}czos et~al., 2013]{poczos2013distribution}
P{\'o}czos, B., Singh, A., Rinaldo, A., and Wasserman, L.~A. (2013).
\newblock Distribution-free distribution regression.
\newblock In {\em AISTATS}, pages 507--515.

\bibitem[Preda, 2007]{preda2007regression}
Preda, C. (2007).
\newblock Regression models for functional data by reproducing kernel hilbert
  spaces methods.
\newblock {\em Journal of statistical planning and inference}, 137(3):829--840.

\bibitem[Ramsay and Silverman, 2007]{ramsay2007applied}
Ramsay, J.~O. and Silverman, B.~W. (2007).
\newblock {\em Applied functional data analysis: methods and case studies}.
\newblock Springer.

\bibitem[Smola et~al., 2007]{smola2007hilbert}
Smola, A., Gretton, A., Song, L., and Sch{\"o}lkopf, B. (2007).
\newblock A hilbert space embedding for distributions.
\newblock In {\em International Conference on Algorithmic Learning Theory},
  pages 13--31. Springer.

\bibitem[Smola and Sch{\"o}lkopf, 1998]{smola1998learning}
Smola, A.~J. and Sch{\"o}lkopf, B. (1998).
\newblock {\em Learning with kernels}, volume~4.
\newblock Citeseer.

\bibitem[Sriperumbudur et~al., 2010a]{sriperumbudur2010relation}
Sriperumbudur, B., Fukumizu, K., and Lanckriet, G. (2010a).
\newblock On the relation between universality, characteristic kernels and rkhs
  embedding of measures.
\newblock In {\em Proceedings of the Thirteenth International Conference on
  Artificial Intelligence and Statistics}, pages 773--780.

\bibitem[Sriperumbudur et~al., 2011]{sriperumbudur2011universality}
Sriperumbudur, B.~K., Fukumizu, K., and Lanckriet, G.~R. (2011).
\newblock Universality, characteristic kernels and rkhs embedding of measures.
\newblock {\em Journal of Machine Learning Research}, 12(Jul):2389--2410.

\bibitem[Sriperumbudur et~al., 2008]{sriperumbudur2008injective}
Sriperumbudur, B.~K., Gretton, A., Fukumizu, K., Lanckriet, G.~R., and
  Sch{\"o}lkopf, B. (2008).
\newblock Injective hilbert space embeddings of probability measures.
\newblock In {\em COLT}, volume~21, pages 111--122.

\bibitem[Sriperumbudur et~al., 2010b]{sriperumbudur2010hilbert}
Sriperumbudur, B.~K., Gretton, A., Fukumizu, K., Sch{\"o}lkopf, B., and
  Lanckriet, G.~R. (2010b).
\newblock Hilbert space embeddings and metrics on probability measures.
\newblock {\em Journal of Machine Learning Research}, 11(Apr):1517--1561.

\bibitem[Steinwart, 2001]{steinwart2001influence}
Steinwart, I. (2001).
\newblock On the influence of the kernel on the consistency of support vector
  machines.
\newblock {\em Journal of machine learning research}, 2(Nov):67--93.

\bibitem[Szab{\'o} et~al., 2016]{szabo2016learning}
Szab{\'o}, Z., Sriperumbudur, B.~K., P{\'o}czos, B., and Gretton, A. (2016).
\newblock Learning theory for distribution regression.
\newblock {\em The Journal of Machine Learning Research}, 17(1):5272--5311.

\bibitem[Vedaldi et~al., 2009]{vedaldi2009multiple}
Vedaldi, A., Gulshan, V., Varma, M., and Zisserman, A. (2009).
\newblock Multiple kernels for object detection.
\newblock In {\em 2009 IEEE 12th international conference on computer vision},
  pages 606--613. IEEE.

\bibitem[Villani, 2008]{villani2008optimal}
Villani, C. (2008).
\newblock {\em Optimal transport: old and new}, volume 338.
\newblock Springer Science \& Business Media.

\bibitem[Whitt, 1976]{whitt1976bivariate}
Whitt, W. (1976).
\newblock Bivariate distributions with given marginals.
\newblock {\em The Annals of statistics}, pages 1280--1289.

\end{thebibliography}

\end{document}